\documentclass[12pt,a4paper]{article}
\usepackage[T1]{fontenc}
\usepackage[utf8]{inputenc}
\usepackage{geometry}
\usepackage{amsmath,amssymb,amsthm,mathtools}
\usepackage{bm}

\usepackage[dvipsnames,svgnames]{xcolor}
\usepackage{graphicx}
\usepackage{booktabs}
\usepackage{float}
\usepackage{needspace}
\floatstyle{ruled}
\newfloat{algorithm}{tbp}{loa}
\floatname{algorithm}{Algorithm}
\usepackage{placeins}
\usepackage{array}
\usepackage{tabularx}
\usepackage{multirow}
\usepackage{enumitem}

\usepackage{cite}
\usepackage{hyperref}
\hypersetup{
  colorlinks=true,
  linkcolor=MidnightBlue,
  citecolor=MidnightBlue,
  urlcolor=MidnightBlue,
  pdftitle={A Projection-Free Algorithm for Variational Inequalities in Hilbert Spaces with Strong Convergence},
  pdfauthor={R. Diaz Millan}
}
\usepackage[nameinlink,capitalise]{cleveref}

\theoremstyle{plain}
\newtheorem{theorem}{Theorem}[section]
\newtheorem{lemma}[theorem]{Lemma}
\newtheorem{proposition}[theorem]{Proposition}
\newtheorem{corollary}[theorem]{Corollary}

\theoremstyle{definition}
\newtheorem{definition}[theorem]{Definition}
\newtheorem{remark}[theorem]{Remark}
\newtheorem{example}[theorem]{Example}
\newtheorem{assumption}{Assumption}

\crefname{assumption}{Assumption}{Assumptions}
\Crefname{assumption}{Assumption}{Assumptions}
\crefname{remark}{Remark}{Remarks}
\Crefname{remark}{Remark}{Remarks}
\crefname{proposition}{Proposition}{Propositions}
\Crefname{proposition}{Proposition}{Propositions}
\crefname{lemma}{Lemma}{Lemmas}
\Crefname{lemma}{Lemma}{Lemmas}
\crefname{theorem}{Theorem}{Theorems}
\Crefname{theorem}{Theorem}{Theorems}

\newcommand{\R}{\mathbb{R}}
\newcommand{\N}{\mathbb{N}}
\newcommand{\Hi}{\mathcal{H}}
\newcommand{\ip}[2]{\langle #1,\, #2 \rangle}
\newcommand{\norm}[1]{\left\| #1 \right\|}
\newcommand{\dist}{\operatorname{dist}}

\newcommand{\Gph}{\operatorname{Gph}}
\newcommand{\dom}{\operatorname{dom}}

\newcommand{\argmax}{\operatorname*{arg\,max}}

\newcommand{\NC}{N_C}
\newcommand{\weak}{\rightharpoonup}
\newcommand{\pos}[1]{\left[#1\right]_+}

\title{\textbf{A Projection-Free Algorithm for Variational Inequalities in Hilbert Spaces with Strong Convergence}}

\author{R.\ D\'{i}az Mill\'{a}n\\[4pt]
  \small Deakin University, Waurn Ponds, Victoria, Australia\\
  \small \texttt{r.diazmillan@deakin.edu.au}}

\date{Revision: November 2026}

\begin{document}
\maketitle

\begin{abstract}
We study variational inequalities governed by a point-to-set maximal monotone
operator in a real Hilbert space and constrained by a continuous convex
inequality. The Feasible Separation Projection Algorithm (FSPA) combines
an inherited Haugazeau--Minty anchoring construction with subgradient outer
cuts and a classical explicit Slater correction. Its outer update projects
the initial point onto at most three half-spaces. Under a quantified feasible
Minty-oracle condition, the whole sequence converges strongly to the nearest
solution. We derive best-iterate $O(N^{-1/2})$ step and feasibility estimates
and a Minty-gap bound with an averaged oracle-error term. Under strong
monotonicity, the best-iterate distance is $O(N^{-1/4})$ when the averaged
oracle error is $O(N^{-1/2})$. The contribution is the point-to-set oracle
analysis and these quantitative estimates within the projection-free outer
construction; the anchoring geometry and Slater error bound are not new.
The abstract theory is Hilbertian, whereas the certified oracle realisations
provided here are finite-dimensional and may require constrained optimisation
of comparable difficulty to a projection. The original finite-net and
large-scale experiments are compared with published methods. No tested regime combines
certification with a demonstrated cost advantage over $P_C$-based methods,
and the practical forward-candidate variant remains heuristic.
\end{abstract}




\section{Introduction}

Variational inequalities provide a common framework for optimality conditions,
equilibrium models, constrained inclusions, and complementarity problems.  In
this paper we consider the variational inequality
\begin{equation}\label{eq:VIP}
  \text{find } x^*\in C \text{ and } u^*\in T(x^*) \text{ such that }
  \ip{u^*}{x-x^*}\ge0 \qquad \forall x\in C,
\end{equation}
where $\Hi$ is a real Hilbert space, $T:\Hi\rightrightarrows\Hi$ is a
point-to-set maximal monotone operator, and
\begin{equation}\label{eq:Cdef}
  C:=\{x\in\Hi:c(x)\le0\},
\end{equation}
with $c:\Hi\to\R$ convex, continuous, and not necessarily differentiable.
The solution set is denoted by $S^*$, and throughout $S^*\ne\emptyset$.  This
setting includes convex minimisation, saddle-point and monotone-inclusion
models, equilibrium problems, and genuinely point-to-set variational
inequalities.

Variational inequalities governed by multivalued operators have been studied
under monotonicity, pseudomonotonicity, generalised monotonicity, and
point-to-set continuity assumptions; see, for instance,
\cite{AnhKuno2012,AnhMuuStrodiot2009,AnhThangThach2021,BaoKhanh2005,BurachikDiazMillan2020,BurachikIusem2008,DongLuYangHe2017,FangHe2013}.
Projection, cutting-hyperplane, and projection--contraction methods are
available in this setting, including Hilbert-space formulations.  The present
paper belongs to this line of work, but focuses on exact feasible Minty cuts and
best-approximation convergence without evaluating the metric projection onto
$C$.

Metric-projection algorithms are effective when $P_C$ is available in closed
form.  If $C$ is described by a nonsmooth inequality or by the maximum of many
convex functions, evaluating $P_C$ may itself require a substantial convex
programme.  Outer-approximation and relaxed-projection methods replace such
steps by projections onto supporting half-spaces; representative references
include \cite{Fukushima1986,Combettes1997,CensorGibaliReich2011,BelloCruzDiazMillan2016}.
Other splitting and penalty methods for constrained variational inequalities
and monotone inclusions include
\cite{AttouchCzarneckiPeypouquet2011,BelloCruzDiazMillan2015,BotHendrich2013}.

In convex optimisation, Frank--Wolfe and conditional-gradient methods
\cite{FrankWolfe1956,Jaggi2013} avoid projections onto the feasible set by
using a linear minimisation oracle. Here, the term ``projection-free'' refers
to the outer half-space update and Slater restoration; the feasible Minty
oracle has different requirements and may itself be computationally demanding.

The closest predecessor of FSPA is the direct method of Bello Cruz and Iusem
\cite{BelloCruzIusem2009}.  Their Algorithm~A constructs a Minty half-space at
an exactly feasible point, adds the same Haugazeau half-space
\[
  W_k=\{x\in\Hi:\ip{x-x^k}{x^0-x^k}\le0\},
\]
projects $x^0$ onto an intersection containing these sets, and proves strong
convergence to $P_{S^*}(x^0)$ in a Hilbert space.  The finite-dimensional
conditional extragradient algorithms in
\cite{BelloCruzDiazMillanPhan2019} use the same anchoring geometry, together
with feasible line searches and normal vectors.  More generally, the
weak-to-strong anchoring principle is rooted in
\cite{SolodovSvaiter2000,BauschkeCombettes2001}; see also
\cite{Haugazeau1968,Wittmann1992,NakajoTakahashi2003,Xu2002}
for related anchoring constructions, and \cite{BelloCruzIusem2011} for
a strongly convergent nonsmooth-minimisation method.  Consequently, neither the Haugazeau half-space nor
the headline form of the best-approximation conclusion is claimed as new here.

The distinction is instead the removal of $P_C$ from the outer update
and feasibility restoration, conditional on the quantified oracle.  The methods in
\cite{BelloCruzIusem2009,BelloCruzDiazMillanPhan2019} use $P_C$ in a trial
step, a line search, or the anchored update.  Moreover,
\cite{BelloCruzIusem2009} explains why its strongly convergent construction is
restricted to point-to-point operators: arbitrary selections from a
point-to-set operator are not suitable for the underlying Korpelevich-type
step.  FSPA replaces the feasible-set projection in the outer update by a
subgradient half-space and restores feasibility of the operator-evaluation
point by an explicit correction towards a fixed Slater point.  The quantified
Minty-gap oracle condition then supplies a workable point-to-set oracle.  This
removes the Armijo search and uniform-continuity hypothesis of the direct method,
at the price of this explicit oracle condition.

Feasible separating hyperplanes also appear in the projection method of
\cite{SolodovSvaiter1999}. The exact feasibility of the graph point is essential.  If $(y,v)\in\Gph(T)$
and $y\in C$, then the half-space $\{x:\ip{v}{x-y}\le0\}$ contains $S^*$.
This conclusion can fail for every infeasible $y$, however close it is to $C$.
Given a fixed point $s$ satisfying $c(s)<0$, the Slater correction moves a trial
point $\widetilde y$ along the segment joining it to $s$ and produces a
feasible graph point using only evaluations of $c$ and a convex combination.

At iteration $k$, FSPA constructs
\begin{align*}
 C_k&:=\{x:c(x^k)+\ip{g^k}{x-x^k}\le0\},\\
 H_k&:=\{x:\ip{v^k}{x-y^k}\le0\},\\
 W_k&:=\{x:\ip{x-x^k}{x^0-x^k}\le0\},
\end{align*}
where $C_k=\Hi$ when $c(x^k)\le0$, and otherwise
$g^k\in\partial c(x^k)$;
$y^k\in C$, and $v^k\in T(y^k)$.  The update is
\begin{equation}\label{eq:mainupdateintro}
  x^{k+1}=P_{C_k\cap H_k\cap W_k}(x^0).
\end{equation}
Thus the outer update projects only onto an intersection of at most three
explicit half-spaces.  Nevertheless, a certified construction of $H_k$ may
require a constrained oracle over $C$.  Throughout this paper,
``projection-free'' therefore means projection-free with respect to the
original feasible set in the outer update and Slater restoration; it does not
mean that every certified Minty oracle is computationally cheaper than $P_C$.

The contribution is the quantified feasible-oracle framework for
point-to-set operators and its convergence and residual analysis. It combines
classical subgradient cuts, Slater restoration and Haugazeau anchoring to
obtain a $P_C$-free outer update. Under the oracle condition, the iterates
converge strongly to the nearest solution. The analysis also gives
best-iterate step and feasibility bounds, a Minty-gap bound retaining the
averaged oracle error, and a distance bound under strong monotonicity.
The finite-dimensional oracle constructions and numerical comparisons
make the certification requirements and their computational cost explicit.

The remainder of the paper is organised as follows.  \Cref{sec:prelim}
collects projection facts and the Minty representation of the solution set.
\Cref{sec:algorithm} states the Slater correction, the oracle, and FSPA.
\Cref{sec:convergence} proves strong convergence and the rate results.
\Cref{sec:strongmonotone} treats strong monotonicity.  \Cref{sec:oracle}
discusses oracle realisations and limitations.  \Cref{sec:positioning} gives a
direct comparison with the closest methods, and \Cref{sec:numerics} reports a
certified disk experiment together with the original large-scale tests
and published comparison methods.

\section{Preliminaries}\label{sec:prelim}

Throughout, \(\Hi\) is a real Hilbert space with inner product
\(\ip{\cdot}{\cdot}\) and norm \(\norm{\cdot}\).  If \(S\subseteq \Hi\) is
nonempty, closed and convex, then \(P_S\) denotes the metric projection onto
\(S\).  The subdifferential of a proper convex function \(f\) at \(x\) is
\(\partial f(x)\).  We use standard facts on projections and monotone operators;
see, for example, \cite{BauschkeCombettes2011,Zarantonello1971}.

\begin{lemma}\label{lem:proj}
Let \(S\subseteq \Hi\) be nonempty, closed and convex.  Then, for every
\(p\in\Hi\), \(x=P_S(p)\) if and only if
\begin{equation}\label{eq:projineq}
  \ip{p-x}{z-x}\le0\qquad \forall z\in S.
\end{equation}
Moreover, if \(x=P_S(p)\), then
\begin{equation}\label{eq:pydist}
  \norm{x-z}^2\le \norm{p-z}^2-\norm{p-x}^2
  \qquad \forall z\in S.
\end{equation}
\end{lemma}

\Needspace{10\baselineskip}
\begin{lemma}\label{lem:haugazeau}
Let \(S\subseteq \Hi\) be nonempty, closed and convex and let
\(\bar x=P_S(x^0)\).  Suppose that
\begin{equation}\label{eq:Wcontain}
  S\subseteq W(x):=\{z\in\Hi:\ip{z-x}{x^0-x}\le0\}.
\end{equation}
Then
\begin{equation}\label{eq:hball}
  x\in B\!\left[\frac{x^0+\bar x}{2},\frac{\norm{x^0-\bar x}}{2}\right].
\end{equation}
In particular, \(\norm{x-x^0}\le \norm{x^0-\bar x}\).
\end{lemma}

\begin{proof}
The same geometric argument is used by Bello Cruz, D\'{i}az Mill\'{a}n
and Phan~\cite{BelloCruzDiazMillanPhan2019}. We include its
Hilbert-space proof for completeness.  Since \(\bar x\in S\subseteq W(x)\), we have
\(\ip{\bar x-x}{x^0-x}\le0\).  Expanding this inequality gives
\begin{equation*}
  \left\|x-\frac{x^0+\bar x}{2}\right\|^2
  \le \frac{1}{4}\norm{x^0-\bar x}^2,
\end{equation*}
which is \eqref{eq:hball}.  The final inequality follows directly from
\eqref{eq:hball}: the centre of the ball is \((x^0+\bar x)/2\) and its radius
is \(\norm{x^0-\bar x}/2\).
\end{proof}

\begin{lemma}\label{lem:kadecklee}
Let \(z^k\) be a sequence in a Hilbert space \(\Hi\) and let \(z\in\Hi\).
If \(z^k\weak z\) and \(\norm{z^k}\to\norm{z}\), then \(z^k\to z\)
strongly.
\end{lemma}

\begin{proof}
By weak convergence, \(\ip{z^k}{z}\to\norm{z}^2\).  Hence
\[
  \norm{z^k-z}^2=\norm{z^k}^2-2\ip{z^k}{z}+\norm{z}^2\to0.
\]
This is the Radon--Riesz (Kadec--Klee) property of Hilbert spaces; see, for
example, \cite{BauschkeCombettes2011}.
\end{proof}

The following lemma is the key reason that the Minty half-space must be built at
an exactly feasible point.

\begin{lemma}\label{lem:ophalfspace}
Let \((y,v)\in\Gph(T)\) with \(y\in C\), and define
\begin{equation}\label{eq:Hyv}
  H(y,v):=\{x\in\Hi:\ip{v}{x-y}\le0\}.
\end{equation}
Then \(S^*\subseteq H(y,v)\).
\end{lemma}

\begin{proof}
Let \(x^*\in S^*\), and choose \(u^*\in T(x^*)\) satisfying
\eqref{eq:VIP}.  Since \(y\in C\), the variational inequality gives
\(\ip{u^*}{y-x^*}\ge0\).  By monotonicity of \(T\),
\(\ip{v-u^*}{y-x^*}\ge0\).  Adding the two inequalities yields
\(\ip{v}{y-x^*}\ge0\), equivalently \(\ip{v}{x^*-y}\le0\).
\end{proof}

For later use define the Minty gap
\begin{equation}\label{eq:gap}
  \Phi(x):=\sup\bigl\{\ip{w}{x-z}: z\in C,\; w\in T(z)\bigr\},
  \qquad \Phi_+(x):=\max\{\Phi(x),0\}.
\end{equation}
The value \(\Phi(x)\) is allowed to be \(+\infty\).

\begin{definition}\label{def:normalcone}
For a nonempty closed convex set \(C\subseteq\Hi\), the normal cone to \(C\) is
\begin{equation*}
  \NC(x):=\{n\in\Hi:\ip{n}{z-x}\le0\ \forall z\in C\},\qquad x\in C,
\end{equation*}
and \(\NC(x)=\emptyset\) if \(x\notin C\).
\end{definition}

\begin{lemma}\label{lem:minty}
Assume that \(T\) is maximal monotone, \(C\subseteq\dom T\) is nonempty, closed
and convex, \(S^*\ne\emptyset\), and \(T+\NC\) is maximal monotone.  Then
\begin{equation}\label{eq:mintyrep}
  S^*=\{x\in C:\Phi(x)\le0\}.
\end{equation}
Consequently, \(S^*\) is closed and convex.
\end{lemma}

\begin{proof}
If \(x^*\in S^*\), then \Cref{lem:ophalfspace} gives
\(\ip{w}{x^*-z}\le0\) for every \((z,w)\in\Gph(T)\) with \(z\in C\).  Hence
\(\Phi(x^*)\le0\).

Conversely, suppose that \(x\in C\) and \(\Phi(x)\le0\).  Then
\begin{equation}\label{eq:mintyineq}
  \ip{w}{x-z}\le0
  \qquad \forall z\in C,
  \quad \forall w\in T(z).
\end{equation}
Let \((z,a)\in\Gph(T+\NC)\).  Then \(z\in C\) and there exist
\(w\in T(z)\) and \(n\in\NC(z)\) such that \(a=w+n\).  Since \(x\in C\),
\(\ip{n}{x-z}\le0\), while \eqref{eq:mintyineq} gives
\(\ip{w}{x-z}\le0\).  Hence
\begin{equation*}
  \ip{a}{x-z}\le0,
  \qquad\text{or equivalently}\qquad
  \ip{0-a}{x-z}\ge0.
\end{equation*}
Thus \((x,0)\) is monotonically related to \(\Gph(T+\NC)\).  Since
\(T+\NC\) is maximal monotone, Minty's maximality criterion implies
\(0\in (T+\NC)(x)\).  Therefore there exist \(u\in T(x)\) and
\(n\in\NC(x)\) with \(u+n=0\).  Since \(n=-u\), the definition of the normal
cone gives
\begin{equation*}
  \ip{u}{z-x}\ge0\qquad \forall z\in C.
\end{equation*}
Thus \(x\in S^*\).  Finally,
\begin{equation*}
  S^*=C\cap\bigcap_{\substack{z\in C\\ w\in T(z)}}
  \{x\in\Hi:\ip{w}{x-z}\le0\},
\end{equation*}
so \(S^*\) is an intersection of closed convex sets.
\end{proof}

\begin{remark}\label{rem:mintyconditions}
The previous proof is the standard Minty argument for variational inequalities;
see Minty's maximality theorem \cite{Minty1962} and, for modern presentations
of set-valued maximal monotone operators, normal-cone sums, and subdifferential
maximal monotonicity,
\cite{BauschkeCombettes2011,BurachikIusem2008,Rockafellar1970}.  The assumption
that \(T+\NC\) is maximal monotone is a convenient regularity condition ensuring
the equivalence between
the Minty and Stampacchia variational inequalities on \(C\).  It is satisfied,
for example, when \(T\) has full domain.  More general domain qualifications are
possible, but then the Minty representation has to be stated relative to
\(C\cap\dom T\).
\end{remark}

\section{Algorithm and standing assumptions}\label{sec:algorithm}

We impose the following assumptions.

\begin{assumption}\label{ass:constraint}
The function $c:\Hi\to\R$ is convex and continuous, $C=\{x:c(x)\le0\}$ is
nonempty, and for every bounded set $B\subset\Hi$ there exists $M_B>0$ such
that
\begin{equation}\label{eq:subgbound}
  \norm{g}\le M_B
  \qquad \forall x\in B \text{ with }c(x)>0,
  \quad \forall g\in\partial c(x).
\end{equation}
\end{assumption}

\begin{remark}\label{rem:A1}
Because $c$ is convex and continuous on all of $\Hi$, its subdifferential is
nonempty at every point.  In finite dimensions, the bounded-subgradient
condition in \eqref{eq:subgbound} is automatic on bounded sets; in an
infinite-dimensional Hilbert space it is a genuine local-boundedness
assumption.
\end{remark}

\begin{assumption}\label{ass:slater}
There exists a fixed and available point $s\in\Hi$ such that
\begin{equation}\label{eq:slater}
  c(s)<0.
\end{equation}
\end{assumption}

\begin{assumption}\label{ass:operator}
The operator $T:\Hi\rightrightarrows\Hi$ is maximal monotone,
$C\subseteq\dom T$, $S^*\ne\emptyset$, and $T+N_C$ is maximal monotone.
This is a convenient qualification for the Minty--Stampacchia equivalence on
$C$; it holds, for example, under the classical full-domain sum condition
\cite{Rockafellar1970,BauschkeCombettes2011}.
\end{assumption}

The following correction and its distance bound are classical; see the
Slater estimate in \cite{BelloCruzIusem2012}. We recall the short argument
because it also specifies the feasible graph-point construction used here.

\begin{lemma}\label{lem:slatercorrection}
Let $s\in\Hi$ satisfy $c(s)<0$.  For $\widetilde y\in\Hi$, define
\begin{equation}\label{eq:lambdak}
  \lambda(\widetilde y):=
  \begin{cases}
  0, & c(\widetilde y)\le0,\\[2mm]
  \dfrac{c(\widetilde y)}{c(\widetilde y)-c(s)},
  & c(\widetilde y)>0,
  \end{cases}
\end{equation}
and
\begin{equation}\label{eq:slatercorrectedy}
  R_s(\widetilde y):=(1-\lambda(\widetilde y))\widetilde y+
  \lambda(\widetilde y)s.
\end{equation}
Then $R_s(\widetilde y)\in C$.
\end{lemma}

\begin{proof}
The conclusion is immediate when $c(\widetilde y)\le0$.  Otherwise,
$0<\lambda(\widetilde y)<1$, and convexity gives
\[
 c(R_s(\widetilde y))
 \le (1-\lambda(\widetilde y))c(\widetilde y)
    +\lambda(\widetilde y)c(s)=0.
\]
\end{proof}

\begin{corollary}\label{cor:slatererrorbound}
For every $x\in\Hi$,
\begin{equation}\label{eq:slatererrorbound}
 \dist(x,C)
 \le \norm{x-R_s(x)}
 =\frac{c_+(x)}{c_+(x)+|c(s)|}\norm{x-s}
 \le \frac{\norm{x-s}}{|c(s)|}c_+(x).
\end{equation}
\end{corollary}

\begin{proof}
Since $R_s(x)\in C$, use $\dist(x,C)\le\norm{x-R_s(x)}
=\lambda(x)\norm{x-s}$ and substitute \eqref{eq:lambdak}.
\end{proof}

\begin{assumption}\label{ass:oracle}
There exist $\sigma\in(0,1]$, a sequence $(\varepsilon_k)\subset\R_+$ with
$\varepsilon_k\to0$, and, for each $k$, an oracle $\mathcal O_k$ defined on
$\Hi$.  When called at any $x\in\Hi$, $\mathcal O_k$ returns a trial point
$\widetilde y\in\Hi$, sets
\[
  y:=R_s(\widetilde y),
\]
and chooses $v\in T(y)$ such that
\begin{equation}\label{eq:oraclecondition}
  \pos{\ip{v}{x-y}}
  \ge \sigma\Phi_+(x)-\varepsilon_k.
\end{equation}
Whenever a sequence $(x^k)$ is bounded, the outputs $(y^k)$ and $(v^k)$ of the
calls $\mathcal O_k(x^k)$ are bounded.
\end{assumption}

\begin{remark}\label{rem:oraclemeaning}
The correction \eqref{eq:lambdak}--\eqref{eq:slatercorrectedy} guarantees exact
feasibility of the graph point without computing $P_C$ or $\dist(\cdot,C)$.
It does not by itself guarantee \eqref{eq:oraclecondition}.  When an
uncorrected trial point satisfies an approximate Minty estimate and the
estimate is transferred through the correction, the loss derived in
\Cref{prop:slaterloss} enters $\varepsilon_k$.  Therefore
$\varepsilon_k\to0$ requires that this correction loss vanish; in particular,
one needs asymptotically feasible trial points,
$\lambda(\widetilde y^k)\norm{\widetilde y^k-s}\to0$, unless the correction is
eventually inactive or a different certificate is available.  The condition
\eqref{eq:oraclecondition} also guarantees that $\Phi_+(x^k)$ is finite along
the iterates.
\end{remark}

We now state the proposed algorithm.  An optional tolerance $\tau\ge0$ may be
used with the computable certificate in Step~5.

\begin{algorithm}[H]
\caption{Feasible Separation Projection Algorithm (FSPA)}\label{alg:main}
\textbf{Input:} $x^0\in\Hi$ and a Slater point $s$ with $c(s)<0$.
An optional stopping tolerance is $\tau\ge0$.
For $k=0,1,2,\ldots$, perform the following steps.
\begin{enumerate}[label=\textbf{Step \arabic*:},leftmargin=4.5em]
\item If $c(x^k)>0$, choose $g^k\in\partial c(x^k)$ and set
\[
 C_k=\{x\in\Hi:c(x^k)+\ip{g^k}{x-x^k}\le0\}.
\]
Otherwise set $C_k=\Hi$.
\item Call $\mathcal O_k$ at $x^k$ to obtain $\widetilde y^k$,
$y^k=R_s(\widetilde y^k)$ and $v^k\in T(y^k)$ satisfying
Assumption~\ref{ass:oracle}. Set
\[
 H_k=\{x\in\Hi:\ip{v^k}{x-y^k}\le0\}.
\]
\item Set
\[
 W_k=\{x\in\Hi:\ip{x-x^k}{x^0-x^k}\le0\}.
\]
\item Compute
\[
 x^{k+1}=P_{C_k\cap H_k\cap W_k}(x^0).
\]
The intersection is nonempty by \Cref{lem:containment}.
\item If $\sigma$ and $\varepsilon_k$ are known, compute
\[
 \mathcal E_k=\max\left\{c_+(x^k),
 \frac{\norm{v^k}\norm{x^{k+1}-x^k}+\varepsilon_k}{\sigma}\right\}.
\]
If a tolerance was prescribed and $\mathcal E_k\le\tau$, return $x^k$ and stop.
\end{enumerate}
\end{algorithm}

\begin{remark}\label{rem:halfspaceqp}
The set $C_k\cap H_k\cap W_k$ is an intersection of at most three closed
half-spaces.  In finite dimensions, the projection of $x^0$ onto this set is
\begin{equation}\label{eq:smallqp}
  \min_x\frac12\norm{x-x^0}^2
  \quad\text{subject to at most three linear inequalities.}
\end{equation}
It can be computed exactly by enumerating the active subsets of the three
constraints or by a standard small quadratic-program solver.  Dykstra's method
is another possibility but is generally asymptotic rather than a finite exact
procedure.  No metric projection onto $C$ occurs in this outer update.
\end{remark}

\section{Convergence analysis}\label{sec:convergence}

\begin{lemma}\label{lem:containment}
For every $k\in\N$,
\begin{equation}\label{eq:containCH}
  S^*\subseteq C_k\cap H_k,
\end{equation}
and
\begin{equation}\label{eq:containCHW}
  S^*\subseteq C_k\cap H_k\cap W_k.
\end{equation}
Consequently, every projection in Step~4 of \Cref{alg:main} is well defined.
\end{lemma}

\begin{proof}
Let $x^*\in S^*$.  If $c(x^k)\le0$, then $C_k=\Hi$.  If $c(x^k)>0$, then
$g^k\ne0$: otherwise $0\in\partial c(x^k)$ would make $x^k$ a global
minimiser of $c$, contradicting $c(s)<0<c(x^k)$.  The subgradient inequality
gives
\[
 c(x^*)\ge c(x^k)+\ip{g^k}{x^*-x^k}.
\]
Since $c(x^*)\le0$, this proves $x^*\in C_k$.  The corrected point satisfies
$y^k\in C$, so \Cref{lem:ophalfspace} gives $x^*\in H_k$.  This proves
\eqref{eq:containCH}.

We prove \eqref{eq:containCHW} by induction.  Since $W_0=\Hi$, it holds at
$k=0$.  If it holds at $k$, the projection characterisation applied to
$x^{k+1}=P_{C_k\cap H_k\cap W_k}(x^0)$ gives
\[
  \ip{x^*-x^{k+1}}{x^0-x^{k+1}}\le0
  \qquad\forall x^*\in S^*,
\]
so $S^*\subseteq W_{k+1}$.  Combining this with \eqref{eq:containCH} at
$k+1$ proves the induction step.
\end{proof}

\begin{lemma}\label{lem:boundedreg}
Let $\bar x=P_{S^*}(x^0)$ and $\rho=\norm{x^0-\bar x}$.  Then
\begin{equation}\label{eq:boundedball}
  x^k\in B\!\left[\frac{x^0+\bar x}{2},\frac{\rho}{2}\right]
  \qquad\forall k\in\N.
\end{equation}
Moreover,
\begin{equation}\label{eq:mononorm}
  0\le\norm{x^{k+1}-x^k}^2
  \le\norm{x^{k+1}-x^0}^2-\norm{x^k-x^0}^2,
\end{equation}
and hence
\begin{equation}\label{eq:regular}
  \norm{x^{k+1}-x^k}\to0.
\end{equation}
\end{lemma}

\begin{proof}
By \Cref{lem:containment}, $S^*\subseteq W_k$.  Applying
\Cref{lem:haugazeau} with $S=S^*$ and $x=x^k$ gives
\eqref{eq:boundedball}.  Since $x^{k+1}\in W_k$,
\[
 0\ge\ip{x^{k+1}-x^k}{x^0-x^k}
 =\frac12\left(\norm{x^{k+1}-x^k}^2
 -\norm{x^{k+1}-x^0}^2+\norm{x^k-x^0}^2\right),
\]
which is \eqref{eq:mononorm}.  Thus $(\norm{x^k-x^0})$ is nondecreasing and,
by \eqref{eq:boundedball}, bounded above by $\rho$.  Letting $k\to\infty$ in
\eqref{eq:mononorm} proves \eqref{eq:regular}.
\end{proof}

\begin{lemma}\label{lem:feasgap}
The sequence generated by \Cref{alg:main} satisfies
\begin{equation}\label{eq:cpluszero}
  c_+(x^k)\to0,
  \qquad \dist(x^k,C)\to0,
\end{equation}
and
\begin{equation}\label{eq:gapzero}
  \Phi_+(x^k)\to0.
\end{equation}
\end{lemma}

\begin{proof}
If $c(x^k)>0$, membership $x^{k+1}\in C_k$ gives
\begin{equation}\label{eq:ccontrol}
  0<c(x^k)\le\norm{g^k}\norm{x^{k+1}-x^k}.
\end{equation}
The iterates are bounded by \Cref{lem:boundedreg}, so Assumption~\ref{ass:constraint}
and \eqref{eq:regular} imply $c_+(x^k)\to0$.  Boundedness of $(x^k)$ and
\eqref{eq:slatererrorbound} then give $\dist(x^k,C)\to0$.

Since $x^{k+1}\in H_k$,
\[
 \ip{v^k}{x^k-y^k}
 \le\ip{v^k}{x^k-x^{k+1}}
 \le\norm{v^k}\norm{x^{k+1}-x^k}.
\]
Assumption~\ref{ass:oracle} makes $(v^k)$ bounded.  Therefore the positive part
above tends to zero, and \eqref{eq:oraclecondition} yields
$\Phi_+(x^k)\to0$.
\end{proof}

\begin{lemma}\label{lem:rate}
Let $\bar x=P_{S^*}(x^0)$, $\rho=\norm{x^0-\bar x}$, and
\[
 M_T:=\sup_k\norm{v^k}<+\infty,
 \qquad D_s:=\sup_k\norm{x^k-s}<+\infty.
\]
Let $M_c$ be a bound for the subgradients occurring at infeasible iterates,
with the convention $M_c:=0$ if every iterate is feasible.  These constants are
trajectory-dependent and need not be known a priori.  For every $N\ge1$,
\begin{align}
  \min_{0\le i<N}\norm{x^{i+1}-x^i}
  &\le\frac{\rho}{\sqrt N}, \label{eq:rate-step}\\
  \min_{0\le i<N}c_+(x^i)
  &\le\frac{M_c\rho}{\sqrt N}, \label{eq:rate-feas}\\
  \min_{0\le i<N}\dist(x^i,C)
  &\le\frac{D_sM_c\rho}{|c(s)|\sqrt N}, \label{eq:rate-dist}\\
  \min_{0\le i<N}\Phi_+(x^i)
  &\le\frac{M_T\rho}{\sigma\sqrt N}
  +\frac{1}{\sigma N}\sum_{i=0}^{N-1}\varepsilon_i. \label{eq:rate-minty}
\end{align}
In particular, the last error term vanishes whenever $\varepsilon_i\to0$.
For every $0\le m<N$,
\begin{equation}\label{eq:rate-tail}
  \min_{m\le i<N}\Phi_+(x^i)
  \le\frac{M_T}{\sigma}
  \left(\frac{\rho^2-\norm{x^m-x^0}^2}{N-m}\right)^{1/2}
  +\frac1\sigma\max_{m\le i<N}\varepsilon_i.
\end{equation}
\end{lemma}

\begin{proof}
Summing \eqref{eq:mononorm} from $i=0$ to $N-1$ gives
\[
 \sum_{i=0}^{N-1}\norm{x^{i+1}-x^i}^2\le\rho^2,
\]
which proves \eqref{eq:rate-step}.  Equation \eqref{eq:ccontrol} gives
$c_+(x^i)\le M_c\norm{x^{i+1}-x^i}$, proving \eqref{eq:rate-feas}; combining
this with \eqref{eq:slatererrorbound} and $\norm{x^i-s}\le D_s$ gives
\eqref{eq:rate-dist}.

For every $i$,
\begin{equation}\label{eq:gapstepbound}
 \Phi_+(x^i)\le\frac{M_T}{\sigma}\norm{x^{i+1}-x^i}
 +\frac{\varepsilon_i}{\sigma}.
\end{equation}
The minimum is no larger than the average, and Cauchy--Schwarz gives
\[
 \frac1N\sum_{i=0}^{N-1}\norm{x^{i+1}-x^i}
 \le\frac{\rho}{\sqrt N},
\]
which proves \eqref{eq:rate-minty}.  Summing \eqref{eq:mononorm} from $m$ to
$N-1$ and applying the same argument on the tail proves
\eqref{eq:rate-tail}.
\end{proof}

\begin{theorem}\label{thm:mainstrong}
Suppose that Assumptions~\ref{ass:constraint}--\ref{ass:oracle} hold.  Then the
sequence $(x^k)$ generated by \Cref{alg:main} converges strongly to
\[
  \bar x=P_{S^*}(x^0).
\]
\end{theorem}

\begin{proof}
The sequence is bounded by \Cref{lem:boundedreg}; let
$x^{k_j}\rightharpoonup\widehat x$.  From \eqref{eq:cpluszero} and weak lower
semicontinuity of $c$, one has $\widehat x\in C$.  The family in
\eqref{eq:gap} is nonempty because $C\subseteq\dom T$ and $C\ne\emptyset$.
The Minty gap is a pointwise supremum of weakly continuous affine functions,
so it is weakly lower semicontinuous.  Hence \eqref{eq:gapzero} gives
$\Phi(\widehat x)\le0$, and \Cref{lem:minty} yields
$\widehat x\in S^*$.

Set $d_k=\norm{x^k-x^0}$.  By \eqref{eq:mononorm}, $(d_k)$ is nondecreasing,
and \eqref{eq:boundedball} gives $d_k\le\rho$.  Thus $d_k\to\ell\le\rho$.
Since $\widehat x\in S^*$,
\[
 \rho\le\norm{\widehat x-x^0}\le\liminf_j d_{k_j}=\ell\le\rho.
\]
Therefore every weak cluster point equals $P_{S^*}(x^0)=\bar x$, so
$x^k\rightharpoonup\bar x$ and
$\norm{x^k-x^0}\to\norm{\bar x-x^0}$.  The Radon--Riesz property of Hilbert
spaces gives $x^k\to\bar x$ strongly.
\end{proof}

\begin{example}\label{ex:infeasiblecut}
Let $\Hi=\R$, $C=(-\infty,0]$, and $T(x)=x-1$.  The unique solution is
$x^*=0$.  For any $\varepsilon\in(0,1)$, the infeasible point $y=\varepsilon$ gives
$T(y)=\varepsilon-1<0$ and
\[
 H(y,T(y))=\{x:(\varepsilon-1)(x-\varepsilon)\le0\}=[\varepsilon,+\infty),
\]
which excludes $x^*$, however small $\varepsilon$ is.
\end{example}

\begin{corollary}\label{cor:stopping}
When the certification constants $\sigma$ and $\varepsilon_k$ are
known, the quantity $\mathcal E_k$ in Step~5 is a computable upper bound for both
$c_+(x^k)$ and $\Phi_+(x^k)$.  If $x^{k+1}=x^k$, then
\begin{equation}\label{eq:stallcertificate}
  c(x^k)\le0,
  \qquad \Phi_+(x^k)\le\frac{\varepsilon_k}{\sigma}.
\end{equation}
Thus a stalled iterate is an exact solution for an exact oracle.
The practical forward-candidate variant does not in general know these
constants and cannot evaluate the second component as a certified bound.
\end{corollary}

\begin{proof}
Since $x^{k+1}\in H_k$,
\[
 \ip{v^k}{x^k-y^k}
 =\ip{v^k}{x^k-x^{k+1}}+\ip{v^k}{x^{k+1}-y^k}
 \le \ip{v^k}{x^k-x^{k+1}}
 \le \norm{v^k}\norm{x^{k+1}-x^k}.
\]
Taking positive parts and using \eqref{eq:oraclecondition} gives
\[
 \Phi_+(x^k)
 \le \frac{\norm{v^k}\norm{x^{k+1}-x^k}+\varepsilon_k}{\sigma},
\]
which proves the asserted upper bound.  If $x^{k+1}=x^k$, membership in $C_k$
rules out $c(x^k)>0$, while membership in $H_k$ makes the positive part in
\eqref{eq:oraclecondition} vanish.
\end{proof}

\section{Strongly monotone point-to-set operators}\label{sec:strongmonotone}

\begin{definition}
The operator $T:\Hi\rightrightarrows\Hi$ is $\mu$-strongly monotone if
$\mu>0$ and
\begin{equation}\label{eq:strongmonotone}
  \ip{u-v}{x-y}\ge\mu\norm{x-y}^2
  \qquad\forall(x,u),(y,v)\in\Gph(T).
\end{equation}
\end{definition}

\begin{proposition}\label{prop:strongunique}
If $T$ is strongly monotone, then \eqref{eq:VIP} has at most one solution.
Consequently, if $S^*\ne\emptyset$, then $S^*=\{x^*\}$.
\end{proposition}

\begin{proof}
Let $x^*,y^*\in S^*$ and choose corresponding $u^*\in T(x^*)$ and
$v^*\in T(y^*)$.  Testing the two variational inequalities at the opposite
solutions and adding gives
$\ip{u^*-v^*}{x^*-y^*}\le0$.  Strong monotonicity gives the reverse lower
bound $\mu\norm{x^*-y^*}^2$, hence $x^*=y^*$.
\end{proof}

\begin{corollary}\label{cor:strongunique}
Under Assumptions~\ref{ass:constraint}--\ref{ass:oracle}, if $T$ is strongly
monotone, then $x^k$ converges strongly to the unique solution $x^*$.
\end{corollary}

\begin{theorem}\label{thm:quadraticgrowth}
Suppose that $T$ is $\mu$-strongly monotone and $S^*=\{x^*\}$.  Then
\begin{equation}\label{eq:quadraticgrowth}
  \Phi_+(x)\ge\frac{\mu}{4}\norm{x-x^*}^2
  \qquad\forall x\in C.
\end{equation}
The constant $\mu/4$ is sharp.
\end{theorem}

\begin{proof}
Choose $u^*\in T(x^*)$ satisfying \eqref{eq:VIP}.  For $x\in C$ and
$t\in(0,1)$, let $z_t=x^*+t(x-x^*)\in C$ and choose $w_t\in T(z_t)$.
Strong monotonicity gives
\[
 \ip{w_t-u^*}{z_t-x^*}\ge\mu t^2\norm{x-x^*}^2,
\]
and therefore
\[
 \ip{w_t}{x-x^*}\ge\ip{u^*}{x-x^*}+\mu t\norm{x-x^*}^2
 \ge\mu t\norm{x-x^*}^2.
\]
Consequently,
\[
 \Phi(x)\ge\ip{w_t}{x-z_t}
 \ge\mu t(1-t)\norm{x-x^*}^2.
\]
Taking $t=1/2$ proves \eqref{eq:quadraticgrowth}.  Sharpness follows from
$T=I$ and $C=[-1,1]\subset\R$, where $x^*=0$ and $\Phi(x)=x^2/4$.
\end{proof}

\begin{theorem}\label{thm:distancerate}
Assume the hypotheses of \Cref{thm:mainstrong}, suppose that $T$ is
$\mu$-strongly monotone, and assume $\Phi_+(s)<+\infty$. The last condition
holds automatically if $C$ is bounded and $T$ is bounded on bounded subsets
of $C$: in that case
\[
 \Phi_+(s)\le
 \left(\sup_{z\in C,\,w\in T(z)}\norm{w}\right)
 \sup_{z\in C}\norm{s-z}<+\infty.
\]
Let $D_s,M_c,M_T$ be as in \Cref{lem:rate}, put $a=|c(s)|$, and define
$\overline\varepsilon_N=N^{-1}\sum_{i=0}^{N-1}\varepsilon_i$. Then
\begin{equation}\label{eq:distancerate1}
\begin{split}
 \min_{0\le i<N}\norm{x^i-x^*}
 &\le \frac{2}{\sqrt\mu}
 \left[
 \left(\frac{M_T}{\sigma}+\frac{\Phi_+(s)M_c}{a}\right)
 \frac{\rho}{\sqrt N}+\frac{\overline\varepsilon_N}{\sigma}
 \right]^{1/2}\\
 &\quad+\frac{D_sM_c\rho}{a\sqrt N}.
\end{split}
\end{equation}
In particular, $\overline\varepsilon_N=O(N^{-1/2})$ implies a best-iterate
$O(N^{-1/4})$ distance bound; it suffices that
$\varepsilon_i=O((i+1)^{-1/2})$. The finite-horizon inequality also holds
without assuming $\varepsilon_i\to0$, provided the feasible cuts, oracle
inequalities and stated bounds hold on that horizon. A fixed positive error
therefore gives an error floor, not an asymptotically vanishing rate.
\end{theorem}

\begin{proof}
Set $d_i=\norm{x^{i+1}-x^i}$ and
$\widehat x^i=R_s(x^i)=(1-\theta_i)x^i+\theta_i s$, where
$\theta_i=\lambda(x^i)$. By \Cref{thm:quadraticgrowth}, convexity of
$\Phi_+$, \eqref{eq:ccontrol}, and \eqref{eq:gapstepbound},
\[
 \frac{\mu}{4}\norm{\widehat x^i-x^*}^2
 \le \Phi_+(\widehat x^i)
 \le \Phi_+(x^i)+\theta_i\Phi_+(s),
 \qquad \theta_i\le\frac{M_c d_i}{a}.
\]
Consequently,
\[
 \norm{x^i-x^*}
 \le \frac{2}{\sqrt\mu}
 \sqrt{\left(\frac{M_T}{\sigma}+\frac{\Phi_+(s)M_c}{a}\right)d_i
       +\frac{\varepsilon_i}{\sigma}}
 +\frac{D_sM_c}{a}d_i.
\]
The minimum is no larger than the average. Apply Jensen's inequality to the
square root and use
$N^{-1}\sum_{i<N}d_i\le\rho/\sqrt N$, which follows from
$\sum_{i<N}d_i^2\le\rho^2$. This proves \eqref{eq:distancerate1}, including
its finite-horizon assertion. Finally,
$\sum_{i<N}(i+1)^{-1/2}\le2\sqrt N$ gives the stated sufficient error decay.
\end{proof}

\begin{remark}\label{rem:tikhonov}
For every $\varepsilon>0$, the regularisation $T+\varepsilon I$ is
$\varepsilon$-strongly monotone.  A variable Tikhonov parameter may be used to
select a distinguished solution of the unregularised problem, but this requires
a separate tracking analysis and is not used here.
\end{remark}

\section{Realisation and limitations of the feasible Minty oracle}\label{sec:oracle}

The convergence theorem is stated in a Hilbert space, but the certified
constructions below are finite-dimensional unless a particular application
supplies its own Hilbert-space oracle.  Thus the abstract theorem is broader
than the implementable certified examples established in this paper.  This is
a limitation of the present oracle theory, not of the outer half-space update.

\subsection{Loss caused by the Slater correction}

For a single-valued operator define
\[
  \psi_x(z):=\ip{T(z)}{x-z}.
\]

\begin{lemma}\label{lem:psilipschitz}
Let $X,B\subset\Hi$ be bounded, suppose that $T$ is single-valued and
$L_T$-Lipschitz on $B$, and set
\[
  M_T^B:=\sup_{z\in B}\norm{T(z)},
  \qquad R_{XB}:=\sup\{\norm{x-z}:x\in X,z\in B\}.
\]
Then, for all $x\in X$ and $z_1,z_2\in B$,
\begin{equation}\label{eq:psilipschitz}
 |\psi_x(z_1)-\psi_x(z_2)|
 \le(L_TR_{XB}+M_T^B)\norm{z_1-z_2}.
\end{equation}
\end{lemma}

\begin{proof}
Write
\[
 \psi_x(z_1)-\psi_x(z_2)
 =\ip{T(z_1)-T(z_2)}{x-z_1}+\ip{T(z_2)}{z_2-z_1}
\]
and apply the stated bounds.
\end{proof}

\begin{proposition}\label{prop:slaterloss}
Let $\sigma\in(0,1]$ and $\eta\ge0$.  Let $X$ be a bounded set of possible
iterates and let $B$ contain $s$, $\widetilde y$, and
$y=R_s(\widetilde y)$.  Under the hypotheses of \Cref{lem:psilipschitz},
\begin{equation}\label{eq:slater-loss-bound}
  |\psi_x(\widetilde y)-\psi_x(y)|
  \le(L_TR_{XB}+M_T^B)\lambda(\widetilde y)\norm{\widetilde y-s}.
\end{equation}
If
\[
 \pos{\psi_x(\widetilde y)}\ge\sigma\Phi_+(x)-\eta,
\]
then the corrected point satisfies the oracle inequality with
\begin{equation}\label{eq:slater-loss-eps}
 \varepsilon
 =\eta+(L_TR_{XB}+M_T^B)\lambda(\widetilde y)
 \norm{\widetilde y-s}.
\end{equation}
\end{proposition}

\begin{proof}
Use \eqref{eq:psilipschitz},
$\norm{\widetilde y-y}=\lambda(\widetilde y)\norm{\widetilde y-s}$, and
$|a_+-b_+|\le|a-b|$.
\end{proof}

Equation \eqref{eq:slater-loss-eps} makes the asymptotic requirement in
\Cref{rem:oraclemeaning} explicit.  If $\eta_k\to0$, then the transferred
oracle-error bounds vanish provided that
$\lambda(\widetilde y^k)\norm{\widetilde y^k-s}\to0$. In particular,
this condition holds when the correction is eventually inactive.

\subsection{Exact and finite-net oracles}

For the remainder of this subsection let $\Hi=\R^n$, let $C$ be compact and
convex, and let $T:C\to\R^n$ be single-valued and continuous.

\begin{proposition}\label{prop:compactoracle}
If the oracle returns
\begin{equation}\label{eq:exactoracle}
  y^k\in\argmax_{y\in C}\ip{T(y)}{x^k-y},
  \qquad v^k=T(y^k),
\end{equation}
then Assumption~\ref{ass:oracle} holds with $\sigma=1$ and
$\varepsilon_k=0$.  An additive $\varepsilon_k$-maximiser gives the inexact
condition with the same $\sigma$.
\end{proposition}

\begin{proof}
Continuity and compactness give existence, and \eqref{eq:exactoracle} is exactly
the supremum defining $\Phi(x^k)$.
\end{proof}

\begin{proposition}\label{prop:finitenetoracle}
Let $X\subset\R^n$ be bounded and contain the iterates, assume $T$ is
$L_T$-Lipschitz on $C$, and set
\[
  M_T^C:=\sup_{y\in C}\norm{T(y)},
  \qquad R_{XC}:=\sup\{\norm{x-y}:x\in X,y\in C\}.
\]
Let $Y_k\subset C$ be finite and satisfy
$\sup_{y\in C}\dist(y,Y_k)\le\delta_k$, with $\delta_k\to0$.  If
\begin{equation}\label{eq:finitenetoracle}
  y^k\in\argmax_{y\in Y_k}\ip{T(y)}{x^k-y},
  \qquad v^k=T(y^k),
\end{equation}
then Assumption~\ref{ass:oracle} holds with $\sigma=1$ and
\begin{equation}\label{eq:finiteepsilon}
  \varepsilon_k=(L_TR_{XC}+M_T^C)\delta_k.
\end{equation}
\end{proposition}

\begin{proof}
By \Cref{lem:psilipschitz}, $\psi_{x^k}$ is
$(L_TR_{XC}+M_T^C)$-Lipschitz on $C$.  Approximate a maximiser over $C$ by a
point of $Y_k$ and use the maximality of $y^k$ over $Y_k$.
\end{proof}

\begin{remark}\label{rem:finitecost}
A global $\delta_k$-net is a certification device rather than a
dimension-free algorithm.  Its cardinality is at least the covering number of
$C$, typically of order $(\operatorname{diam}C/\delta_k)^n$.  Thus forcing
$\varepsilon_k\to0$ by global nets is exponentially expensive in dimension.
Small adaptive candidate pools may work well in practice, but do not inherit the
global certificate without an additional argument.
\end{remark}

\subsection{Affine monotone oracle and the meaning of projection-free}

\begin{proposition}\label{prop:affineoracle}
Let $C\subset\R^n$ be nonempty, closed, and convex, and let $T(y)=My+q$ with
$M+M^\top\succeq0$.  For fixed $x$, the exact Minty oracle is equivalent to
\begin{equation}\label{eq:affineoraclegeneral}
  \min_{y\in C}
  \left\{\frac12y^\top(M+M^\top)y+(q-M^\top x)^\top y\right\}.
\end{equation}
If $M+M^\top\succ0$, this is exactly the metric projection of
\begin{equation}\label{eq:weightedcentre}
  (M+M^\top)^{-1}(M^\top x-q)
\end{equation}
onto $C$ in the norm induced by $M+M^\top$.
\end{proposition}

\begin{proof}
The identity
\[
 \ip{My+q}{x-y}=q^\top x-
 \left[\frac12y^\top(M+M^\top)y+(q-M^\top x)^\top y\right]
\]
gives \eqref{eq:affineoraclegeneral}; completing the square proves the final
statement.
\end{proof}

\begin{remark}\label{rem:oraclecost}
The exact affine oracle can therefore be a weighted projection in disguise and
may be of the same difficulty class as $P_C$.  More generally, monotonicity of
$T$ does not make $y\mapsto\psi_x(y)$ concave.  For example, $T(y)=y^3$ on
$\R$ is monotone, while for $x=3$,
\[
 \psi_3(y)=3y^3-y^4,
 \qquad \psi_3''(y)=6y(3-2y)>0\quad(0<y<3/2).
\]
Hence even an approximate maximiser may require global optimisation.  The term
``projection-free'' describes the outer FSPA update and the explicit Slater
restoration, not the intrinsic cost of every certified oracle.
\end{remark}

The computable separation
$r_k=\pos{\ip{v^k}{x^k-y^k}}$ can be checked against any certified upper bound
$U_k\ge\Phi_+(x^k)$ through $r_k\ge\sigma U_k-\varepsilon_k$.  Such bounds may
come from a duality gap, the finite-net estimate, or a problem-specific
relaxation.

\section{Positioning with respect to the closest methods}\label{sec:positioning}

The anchoring lineage relevant to FSPA is
\[
\begin{gathered}
 \text{weak-to-strong anchoring principles}
 \longrightarrow \text{Bello Cruz--Iusem}\\
 \longrightarrow \text{conditional extragradient variants}
 \longrightarrow \text{FSPA}.
\end{gathered}
\]
The purpose of \Cref{tab:comparison} is to delimit the incremental contribution,
not to suggest that FSPA dominates methods for which $P_C$ is inexpensive.

\begin{table}[H]
\centering
\caption{Comparison with the closest anchoring and variational-inequality methods.}
\label{tab:comparison}
\small
\renewcommand{\arraystretch}{1.12}
\begin{tabularx}{\textwidth}{>{\raggedright\arraybackslash}p{2.8cm}
>{\raggedright\arraybackslash}p{2.5cm}
>{\raggedright\arraybackslash}p{2.4cm}
>{\raggedright\arraybackslash}p{2.8cm}>{\raggedright\arraybackslash}X}
\toprule
Method & Use of $P_C$ & Operator model & Cut/projection construction & Main feature \\
\midrule
Solodov--Svaiter \cite{SolodovSvaiter2000} and Bauschke--Combettes \cite{BauschkeCombettes2001}
& General anchoring principles, not the present VI update
& Maximal-monotone frameworks
& Problem-dependent
& Weak-to-strong and best-approximation architecture \\
\addlinespace
Bello Cruz--Iusem Algorithm~A \cite{BelloCruzIusem2009}
& Used in the trial point and anchored projection
& Point-to-point monotone operator
& Feasible convex combination involving $P_C$
& Strong convergence to $P_{S^*}(x^0)$ in Hilbert space \\
\addlinespace
Conditional extragradient algorithms \cite{BelloCruzDiazMillanPhan2019}
& Used in trial and/or update steps
& Single-valued, finite-dimensional
& Feasible line searches, with optional normal vectors
& Anchored finite-dimensional convergence \\
\addlinespace
Bello Cruz--Iusem Algorithm~A \cite{BelloCruzIusem2012}
& No $P_C$; subgradient half-space projections
& Point-to-point monotone operator
& Inner feasibility loop and relaxed projected step
& Weak convergence of the prescribed weighted averages \\
\addlinespace
Bello Cruz--D\'{i}az Mill\'{a}n Algorithm~A \cite{BelloCruzDiazMillan2014}
& Two projections onto $C$ per iteration
& Sum of two point-to-set monotone operators
& Two projected subgradient steps
& Weak convergence of the prescribed weighted averages \\
\addlinespace
FSPA
& No $P_C$ in the outer update or Slater restoration
& Point-to-set maximal monotone
& Subgradient cut, explicit Slater correction, quantified Minty oracle
& Strong best-approximation convergence and residual/distance rates \\
\bottomrule
\end{tabularx}
\end{table}

Algorithm~A of \cite{BelloCruzIusem2012} already avoids $P_C$ through
subgradient half-space projections and treats point-to-point operators with
weak convergence of its prescribed weighted averages, whereas FSPA combines
the inherited anchoring with a quantified feasible point-to-set Minty oracle
to obtain strong convergence to $P_{S^*}(x^0)$ under
Assumptions~\ref{ass:constraint}--\ref{ass:oracle}.
Algorithm~A of \cite{BelloCruzDiazMillan2014} already treats sums of
point-to-set operators, but uses two projections onto $C$ per iteration and
establishes weak convergence of its prescribed weighted averages.

Relative to \cite{BelloCruzIusem2009}, the decisive differences are the
elimination of $P_C$ from the outer update and restoration, the point-to-set oracle, the absence of an Armijo line
search and uniform-continuity hypothesis, and the rate estimates.  Relative to
\cite{BelloCruzDiazMillanPhan2019}, the differences are the Hilbert-space
analysis, set-valued operators, and the absence of normal-vector and
feasible-projection machinery.  The price is the strength and possible
computational expense of Assumption~\ref{ass:oracle}, as made explicit in
\Cref{sec:oracle}.

\section{Numerical experiments and comparisons}\label{sec:numerics}

We first report the certified finite-net experiment on the unit disk, followed
by the original nonsmooth and large-scale test problems. The latter use the
practical forward single-candidate variant. At each iteration it forms
\[
 \widetilde y^k=x^k-\alpha_k F(x^k),\qquad
 y^k=R_s(\widetilde y^k),\qquad v^k=F(y^k),
\]
where $F$ is the specified single-valued operator or subgradient selection,
and then performs the FSPA outer update. Its graph point is feasible, but
no quantified Minty-gap certificate is supplied. We therefore call this
implementation \emph{FSPA-forward (heuristic)}. Its experiments illustrate
finite-budget behaviour and do not verify the convergence theorem.

The reported feasibility measure is $c_+(x^k)=\max\{c(x^k),0\}$.
The output change is $\norm{x^{k+1}-x^k}$; for an averaging method it is
the change of its prescribed average. Neither quantity alone certifies
proximity to a solution. Distances are reported when a solution is known.
The problem definitions, dimensions and main FSPA iteration budgets are
retained from the submitted version. The data generators, seeds, initial
points and trial rules follow the available author scripts. All values
reported below come from the documented reimplementation; they are not
claimed to reproduce every submitted numerical value. The supplementary
package gives the complete parameter record and a comparison with the
submitted tables.

\paragraph{Timing.}
Times are medians of five complete runs after one warm-up, using double
precision and one BLAS thread. They include operator and constraint
evaluations, all inner corrections and projections, and trajectory recording.
Finite-net setup is included in total time. The reported per-step time is
the difference between median total and median setup time, divided by the
number of outer steps.
Problem generation, diagnostic gaps and distances, plotting and file output
are excluded. The package records the environment, timing samples and work
counts. All compared methods on a given problem are timed under these rules.

\subsection{Certified affine variational inequality on the unit disk}
\label{subsec:certifiedexperiment}

Let
\begin{equation}\label{eq:certifiedproblem}
 C=\{x\in\R^2:\norm{x}\le1\},\qquad F(x)=Mx+q,
\end{equation}
with
\[
 M=\begin{pmatrix}2&-1\\1&1\end{pmatrix},\qquad
 q=\begin{pmatrix}-3\\0.5\end{pmatrix},\qquad
 x^0=\begin{pmatrix}0.2\\0.8\end{pmatrix}.
\]
The symmetric part is $\operatorname{diag}(2,1)$, so $F$ is
$1$-strongly monotone. The unique solution, computed independently from
the KKT system, is
\[
 x^*=(0.7798023042602009,-0.6260258511199684)^\top.
\]
We use $c(x)=\norm{x}^2-1$ and the Slater point $s=0$.
For mesh size $h$, the candidate set is a Cartesian grid of spacing $h$
inside the disk, augmented by a boundary grid with arc spacing at most $h$
and the origin. Rounded boundary points are scaled into the disk.
At each iteration the oracle chooses the feasible candidate maximising
$\ip{F(y)}{x^k-y}$. The Slater correction is inactive.

A covering radius is $\delta_h=3h/2$: points at least $h$ from the boundary
have a grid neighbour within $h/\sqrt2$, and the remaining points have a
boundary neighbour within $h+h/2$. Put
\[
 \rho=\norm{x^0-x^*},\qquad R_x=\norm{x^0}+\rho,\qquad
 L=\norm M_2,\qquad M_T=L+\norm q.
\]
The finite-net oracle satisfies the quantified inequality with $\sigma=1$
and the fixed error
\begin{equation}\label{eq:diskneterror}
 \varepsilon_h=[L(R_x+1)+M_T]\frac{3h}{2}.
\end{equation}
This is a finite-horizon certificate. A fixed mesh does not satisfy the
vanishing-error assumption of the asymptotic convergence theorem.

The closest baseline remains Algorithm~A of Bello Cruz and
Iusem~\cite{BelloCruzIusem2009}, with fixed admissible parameters
$\beta_k=0.4$ and $\delta=0.5$ and its prescribed dyadic line search.
The projection onto the disk and its intersection with the two anchored
cuts is computed by geometric enumeration. FSPA uses active-set enumeration
for its at most three half-spaces. Constraints are normalised and checked
with tolerance $2\times10^{-12}$. Both methods stop at a step norm of
$10^{-10}$ or 300 iterations; the baseline also recognises a projected
fixed point to tolerance $10^{-13}$. These are numerical stopping rules.
The refined projection and stopping tolerances explain small changes from
the earlier iteration counts; the disk and its four mesh sizes are unchanged.

\begin{table}[tb]
\centering\small
\caption{Certified unit-disk experiment with measured costs.
Total time includes setup; per-step time excludes setup. Diagnostic gaps and
distances are excluded from timing. Candidate values are cached once.}
\label{tab:certifiedexperiment}
\setlength{\tabcolsep}{3pt}
\begin{tabular}{@{}lrrrrrr@{}}
\toprule
Method & Candidates & Steps & Distance & $\Phi_+$ & Step ($\mu$s) & Total (ms)\\
\midrule
FSPA, $h=0.1$ & 382 & 10 & $3.29\times10^{-2}$ & $8.42\times10^{-4}$ & 152.6 & 1.674 \\
FSPA, $h=0.05$ & 1384 & 10 & $2.82\times10^{-2}$ & $6.16\times10^{-4}$ & 164.6 & 1.885 \\
FSPA, $h=0.035$ & 2747 & 11 & $1.32\times10^{-2}$ & $1.36\times10^{-4}$ & 164.5 & 2.154 \\
FSPA, $h=0.025$ & 5278 & 13 & $3.25\times10^{-3}$ & $8.18\times10^{-6}$ & 168.3 & 2.742 \\
Algorithm~A~\cite{BelloCruzIusem2009} & -- & 14 & $1.06\times10^{-6}$ & $8.66\times10^{-13}$ & 130.7 & 1.829 \\
\bottomrule
\end{tabular}
\end{table}

The finite-net values $F(y)$ and $\ip{F(y)}{y}$ are cached once; subsequent
oracle calls evaluate affine candidate scores. Candidate counts therefore
differ from operator-call counts. Setup costs for $h=0.1,0.05,0.035,0.025$
are 0.148, 0.239, 0.345, 0.553\,ms, respectively. The complete counts accompany the
data. In \Cref{tab:certifiedexperiment}, the projection-based method reaches
a smaller distance error with less total time than the finest-net FSPA
run. The coarsest net stops sooner at much poorer accuracy, so the terminal
times do not constitute a comparison at equal accuracy.

\paragraph{Connection with the distance bound.}
For the same four finite-net trajectories, \Cref{fig:certifiedexperiment}
reports $\min_{0\le i<N}\norm{x^i-x^*}$ and the complete bound in
\Cref{thm:distancerate}. We use $\mu=\sigma=1$, $|c(s)|=1$,
$D_s=R_x$, $M_c=2R_x$, the preceding $M_T$, and the independent KKT value
of $\Phi_+(0)$. These are bounds from the problem data, not fitted constants.
The bound retains $\overline\varepsilon_N=\varepsilon_h$ and consequently
has a positive error floor. The finite trajectories are consistent with
this estimate; they do not demonstrate an asymptotic $N^{-1/4}$ decay.
Exact Minty gaps used as diagnostics are evaluated by the scalar KKT
solution of \eqref{eq:affineoraclegeneral}, outside the timed loops.

\begin{figure}[tb]
\centering
\includegraphics[width=\textwidth]{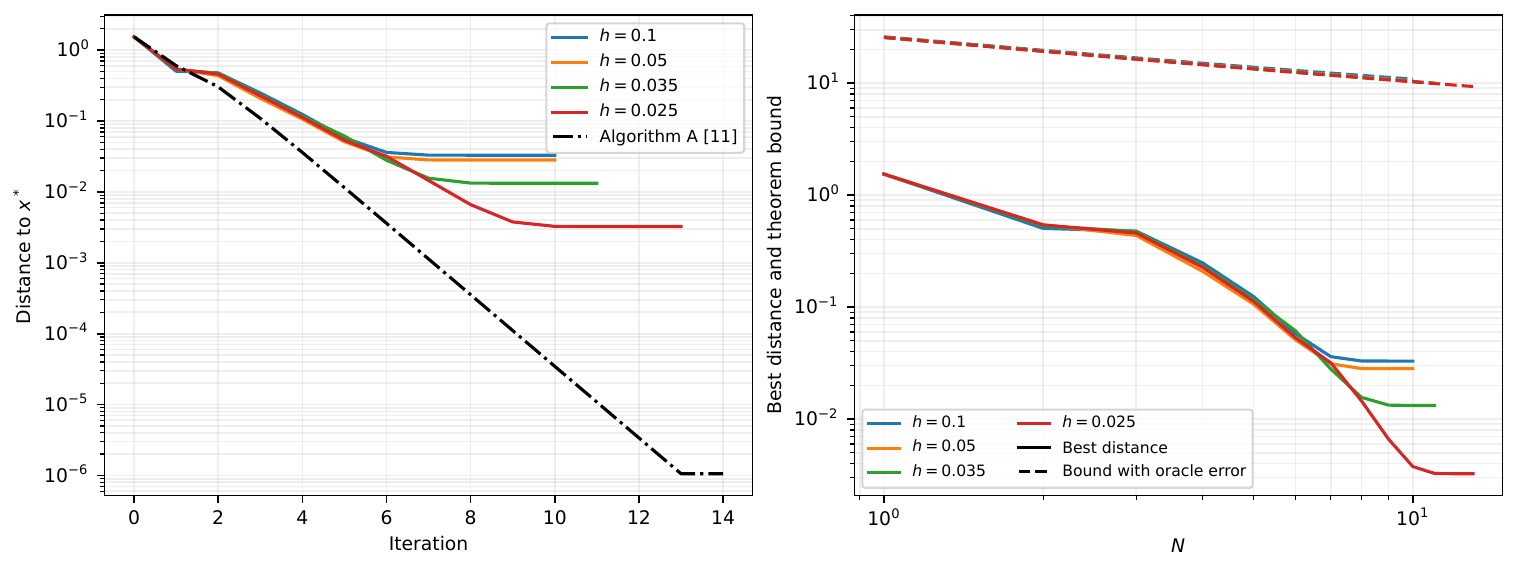}
\caption{The original unit-disk experiment. Left: distance to the independently
computed solution for the four finite nets and Algorithm~A
of~\cite{BelloCruzIusem2009}. Right: best-iterate distances (solid) and
their theoretical bounds including the fixed-net errors (dashed), with
matching colours for each mesh. Curves end at their recorded stopping indices.}
\label{fig:certifiedexperiment}
\end{figure}

\subsection{Benchmark methods and retained FSPA settings}\label{subsec:protocol}

For the smooth-operator tests we use Algorithm~A of Bello Cruz and
Iusem~\cite{BelloCruzIusem2012}. The implementation includes its inner
feasibility loop, operator-norm normalisation, subgradient half-space at
the resulting point and prescribed weighted average. Its control sequence
is $\zeta_k=(k+1)^{-3/4}$, the normalising factor is
$\max\{1,\norm{F(\widetilde y^k)}\}$, and the inner feasibility constant
is $\theta=1$. The inner loop uses the Slater distance estimate in
\Cref{cor:slatererrorbound}; every inner correction is counted.
The auxiliary starting point is shared with FSPA; the prescribed initial
average is zero. We compare the published output of each method and do
not interpret a small change of an average as a stationarity certificate.

For the point-to-set test we use Algorithm~A of Bello Cruz and
D\'{i}az Mill\'{a}n~\cite{BelloCruzDiazMillan2014}, with
$\alpha_k=(k+1)^{-3/4}$. It performs both full projected subgradient steps
for $T_1+T_2$, followed by the prescribed weighted averaging. The second
step is a subgradient step, not a proximal substitution. If the published
fixed-point test identifies a solution, we report that auxiliary point
explicitly and retain both sequences in the data.

The two published algorithms require positive sequences with divergent sum
and summable squares, rather than a unique numerical default. The sequence
$(k+1)^{-3/4}$ meets these conditions; the constants are fixed admissible
choices across instances and were not selected by a tuning sweep. For
background on projected and incremental subgradient methods, see
\cite{Polyak1969,Alber1998,NedicBertsekas2001}.

For FSPA-forward, the rule $\alpha_k=\alpha_0(k+1)^{-\beta}$
uses $(\alpha_0,\beta)=(1.2,0.7)$ for the small max-quadratic example,
$(1,0.75)$ for the large max-quadratic and mixed-norm examples,
$(10,0.75)$ for optimal control, $(0.8,0.75)$ for Cournot--Nash,
and $(1,0.5)$ for the point-to-set example. These coefficients are retained
from the available author scripts; they do not establish the Minty-oracle
condition. The submitted PDF does not specify all these coefficients.
Both methods receive 40 outer iterations on the small max-quadratic
instance, 80 on each other smooth instance, and at most 5000 on the
point-to-set instance. The supplementary code maps the implementations
to the published updates and specifies all numerical tolerances.

\subsection{Test problems}

\paragraph{Max of convex quadratics.}
The first problem is designed to stress projection-based methods.  The feasible
set is given by the nonsmooth max-function
\begin{equation}\label{eq:maxquad-c}
  C=\{x\in\R^n:c(x)\le0\},\qquad
  c(x)=\max_{1\le i\le m}\left\{\frac12 x^\top Q_i x+a_i^\top x+b_i\right\},
\end{equation}
where each \(Q_i\) is symmetric positive semidefinite.  The operator is the
nonsymmetric affine map
\begin{equation}\label{eq:maxquad-operator}
  F(x)=Mx+q,\qquad M=A^\top A+\Omega,
\end{equation}
where \(\Omega^\top=-\Omega\).  The skew part makes the problem a genuine variational
inequality rather than simply the optimality condition of a convex
minimisation problem.  Projection onto \(C\) requires solving a QCQP with \(m\)
quadratic constraints, while FSPA uses only one active quadratic
cut plus the Minty and Haugazeau half-spaces.

\paragraph{Discretised optimal control.}
The second test comes from a Hilbert-space model.  Let the control variable
belong to \(L^2(0,1)\), and define the Volterra state map
\[
  (Ku)(t)=\int_0^t u(\vartheta)\,d\vartheta .
\]
A typical monotone optimal-control operator is
\begin{equation}\label{eq:oc-operator}
  F(u)=\kappa u+K^*(Ku-y_d)+\gamma \Omega u,
\end{equation}
where \(\kappa>0\), \(\Omega\) is a skew-symmetric discretisation and \(y_d\) is a
desired state.  The nonsmooth state constraint is
\begin{equation}\label{eq:oc-constraint}
  c(u)=\|Ku\|_\infty-U\le0.
\end{equation}
The Slater point is \(s=0\), with \(c(0)=-U<0\).  After discretisation with
\(n_{\mathrm{mesh}}=10000\) mesh points, the problem is a large finite-dimensional VI inherited
from the Hilbert-space formulation.

\paragraph{Sparse recovery with mixed norms.}
The third test uses the high-dimensional mixed-norm constraint
\begin{equation}\label{eq:mixed-constraint}
  c(x)=\omega_1\|x\|_1+\omega_\infty\|x\|_\infty-R_{\mathrm{mix}}\le0.
\end{equation}
The operator is a least-squares gradient perturbed by a skew-symmetric term,
\begin{equation}\label{eq:mixed-operator}
  F(x)=\frac1p A^\top(Ax-b)+\mu x+\Omega x,
\end{equation}
with \(\mu>0\) and \(\Omega^\top=-\Omega\).  This example is included as a robustness test
on a nonsmooth high-dimensional geometry; it is not expected to be the most
favourable case for FSPA.

\paragraph{Cournot--Nash equilibrium with shared capacity.}
The fourth test is an equilibrium model with \(N_f\) firms and \(N_m\) markets.
The decision variable is \(x=(x_{ij})\in\mathbb R^{N_fN_m}\), where \(x_{ij}\)
denotes the production of firm \(i\) in market \(j\).  For each market, set
\[
  Q_j(x):=\sum_{i=1}^{N_f}x_{ij}.
\]
We use the affine Cournot--Nash operator
\begin{equation}\label{eq:cournot-operator}
  F_{ij}(x)
  = \chi_{ij}+h_{ij}x_{ij}-a_j+b_jQ_j(x)+b_jx_{ij}+(\Omega x)_{ij},
  \qquad i=1,\ldots,N_f,
  \quad j=1,\ldots,N_m,
\end{equation}
where \(a_j,b_j>0\) define the linear inverse demand
\(p_j(Q)=a_j-b_jQ\), \(\chi_{ij}\) and \(h_{ij}>0\) are the cost parameters, and
\(\Omega^\top=-\Omega\) is an additional skew-symmetric coupling.  The term
\(\chi_{ij}+h_{ij}x_{ij}-a_j+b_jQ_j(x)+b_jx_{ij}\) is the marginal cost minus the
marginal revenue of firm \(i\) in market \(j\), and the skew part preserves
monotonicity of the symmetric component while making the operator nonsymmetric.
The shared nonsmooth capacity constraint is
\begin{equation}\label{eq:cournot-constraint}
  c(x)=\max\left\{\max_{i,j}(-x_{ij}),\;\max_{1\le r\le R}\big((Bx)_r-b^{\mathrm{cap}}_r\big)\right\}\le0,
\end{equation}
where \(x\ge0\) represents nonnegative production and \(Bx\le b^{\mathrm{cap}}\) represents
coupled resource or emission limits.  The Slater point is chosen as a strictly
positive production vector scaled so that \(Bs<b^{\mathrm{cap}}\).  The test uses \(60\) firms,
\(15\) markets and \(25\) shared resource constraints, so \(n=900\).

\paragraph{Point-to-set \(\ell_1\)-subdifferential example.}
The fifth test is included to show that the point-to-set formulation is used in
the computations, and not only in the abstract theory.  We consider
\begin{equation}\label{eq:pointset-operator}
  T(x)=\mu x+\omega\partial\norm{x}_1,\qquad \mu>0,
  \quad \omega>0,
\end{equation}
with the polyhedral inequality
\begin{equation}\label{eq:pointset-constraint}
  c(x)=\max_{1\le r\le m}(a_r^\top x-b_r)\le0,\qquad b_r>0.
\end{equation}
The Slater point is \(s=0\), and \(0\in C\) and \(0\in T(0)\), so the
unique solution is \(x^*=0\).  At an operator evaluation point \(y\), the
implementation uses the selection
\[
  v=\mu y+\omega\xi, \qquad
  \xi_i\in\partial |y_i|,
\]
with \(\xi_i=\operatorname{sign}(y_i)\) when \(y_i\ne0\), and
\(\xi_i=0\) when \(y_i=0\).  This gives a concrete maximal monotone
point-to-set test for the half-space construction.

\FloatBarrier
\subsection{Results on the smooth-operator examples}\label{subsec:largescale}

\Cref{tab:numerical-summary} reports the five submitted smooth-operator
instances with the trial rules stated above and the published half-space method
replacing the unnamed relaxed baseline. The associated default/retuned
comparison is removed. The small exact-projection illustration is also
omitted because its five-step solver-dependent stop was not comparable
with the common outer-iteration budget.

\begin{table}[H]
\centering\small
\caption{Smooth-operator problems with the documented FSPA trial rules and
the published comparison method. Both methods use 40 iterations on the small
quadratic instance and 80 otherwise. FSPA-forward is heuristic;
Algorithm~A of~\cite{BelloCruzIusem2012} reports its weighted average.
``Change'' is the last change of the reported output.}
\label{tab:numerical-summary}
\setlength{\tabcolsep}{4pt}
\begin{tabular}{@{}llrrrr@{}}
\toprule
Problem ($n$) & Method & Iter. & $c_+$ & Change & Time (ms)\\
\midrule
Max quadratics (15) & FSPA-forward & 40 & $0$ & $1.29\times10^{-2}$ & 7.03 \\
 & Algorithm~A~\cite{BelloCruzIusem2012} & 40 & $0$ & $4.19\times10^{-4}$ & 2.10 \\
\addlinespace
Max quadratics (5000) & FSPA-forward & 80 & $0$ & $6.29\times10^{-4}$ & 124.63 \\
 & Algorithm~A~\cite{BelloCruzIusem2012} & 80 & $0$ & $2.32\times10^{-3}$ & 674.81 \\
\addlinespace
Optimal control (10000) & FSPA-forward & 80 & $5.60\times10^{-3}$ & $2.07\times10^{0}$ & 255.63 \\
 & Algorithm~A~\cite{BelloCruzIusem2012} & 80 & $0$ & $2.03\times10^{-2}$ & 33.45 \\
\addlinespace
Mixed norm (1500) & FSPA-forward & 80 & $0$ & $4.46\times10^{-3}$ & 104.71 \\
 & Algorithm~A~\cite{BelloCruzIusem2012} & 80 & $0$ & $2.79\times10^{-3}$ & 36.80 \\
\addlinespace
Cournot--Nash (900) & FSPA-forward & 80 & $0$ & $5.48\times10^{-5}$ & 26.49 \\
 & Algorithm~A~\cite{BelloCruzIusem2012} & 80 & $0$ & $3.49\times10^{-3}$ & 49.36 \\
\bottomrule
\end{tabular}
\end{table}

\paragraph{Reproducibility of the reruns.}
The small max-quadratic trajectory is sensitive to projection roundoff:
the normalised routine gives a final step of $1.29\times10^{-2}$, compared
with the submitted $5.71\times10^{-3}$. For optimal control, the documented
settings reproduce the available standalone record:
$c_+\approx5.60\times10^{-3}$ and final step $2.07$. They do not reproduce
the submitted zero violation and step $3.95\times10^{-2}$, whose run
configuration could not be verified. We report the reproducible replacement
and archive the comparison; no transcription error in the submission is
asserted.

The figures retain the max-quadratic, optimal-control and Cournot--Nash
diagnostics with the published comparator. The former control display-floor
panel and the now identically zero max-quadratic and Cournot feasibility
panels are omitted; the table reports feasibility. Small output changes,
especially changes of a weighted average, are not solution-distance
certificates. These results describe feasibility, movement and measured
work, without establishing an accuracy-adjusted performance ranking.
The point-to-set example below permits a direct solution-distance comparison.

\begin{figure}[tb]
\centering
\includegraphics[width=\textwidth]{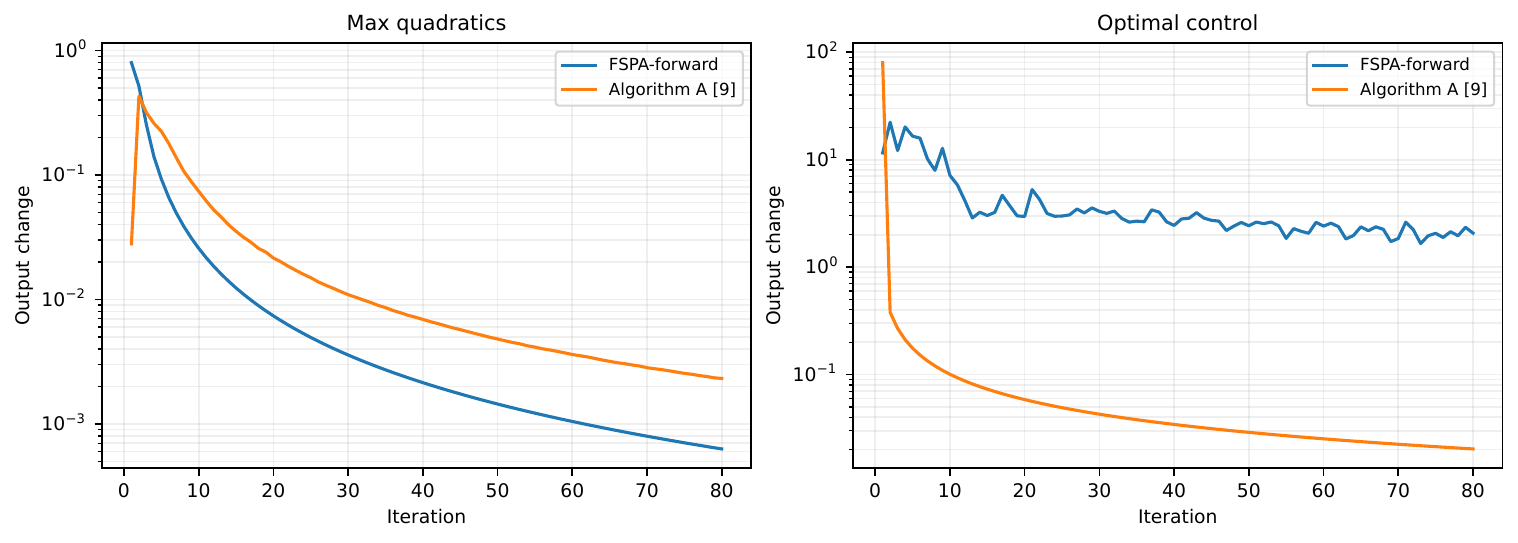}
\caption{Large max-quadratic and optimal-control examples with
Algorithm~A of~\cite{BelloCruzIusem2012} as comparator. Left: max-quadratic
output change. Right: optimal-control output change. FSPA-forward is
heuristic; the comparator reports its weighted average. The feasibility
values are reported in \Cref{tab:numerical-summary}.}
\label{fig:largecontrol}
\end{figure}

\begin{figure}[tb]
\centering
\includegraphics[width=0.72\textwidth]{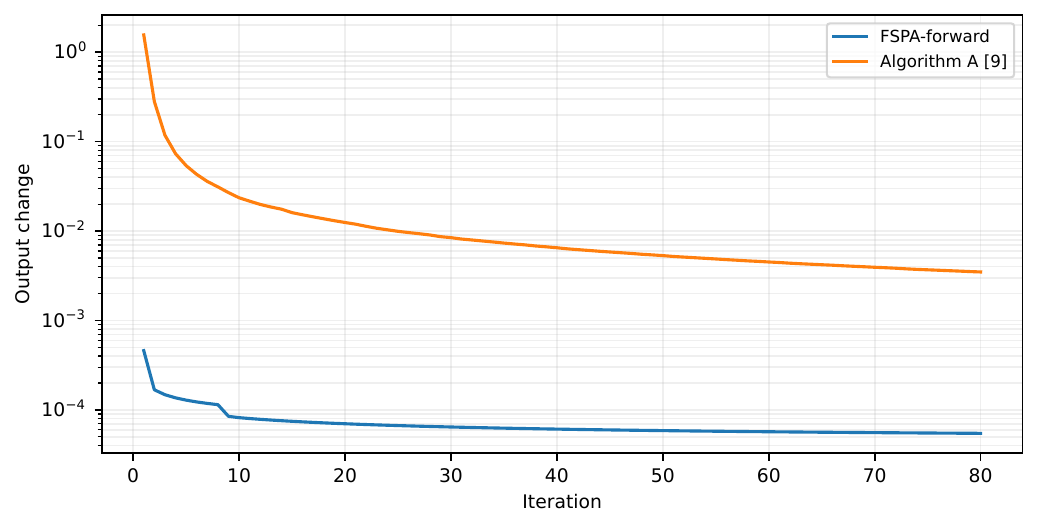}
\caption{Cournot--Nash example: output changes, using the same conventions
as \Cref{fig:largecontrol}. Both recorded sequences have zero constraint
violation, so the feasibility panel is omitted.}
\label{fig:cournot}
\end{figure}
\FloatBarrier

\subsection{Point-to-set comparison}\label{subsec:pointset}

For the original $n=1200$, $m=80$, seed-31 problem, we retain
$\mu=1$, $\omega=0.05$, $b_r=0.8$ and the Gaussian construction
$a_{rj}\sim\mathcal N(0,1/n)$. The initial Gaussian point, with standard
deviation 2, is Slater-corrected; its norm is $11.6297781264$.
The solution is $x^*=0$. The FSPA-forward run uses the original
5000-step budget and trial rule recorded in \Cref{subsec:protocol}.

\begin{table}[H]
\centering\small
\caption{Original point-to-set example with the published splitting method.
The zero is its identified auxiliary solution, not its two-step weighted
average. Times include all projection calls and exclude diagnostic distances.}
\label{tab:pointset}
\setlength{\tabcolsep}{5pt}
\begin{tabular}{@{}lrrrr@{}}
\toprule
Method & Iter. & Distance to $0$ & $c_+$ & Time (ms)\\
\midrule
FSPA-forward (heuristic) & 5000 & $2.29\times10^{-1}$ & $0$ & 1749.353 \\
Algorithm~A~\cite{BelloCruzDiazMillan2014} & 2 & $0$ & $0$ & 0.112 \\
\bottomrule
\end{tabular}
\end{table}

Algorithm~A of~\cite{BelloCruzDiazMillan2014} solves this particular instance
immediately: $T_1=I$ and $\alpha_0=1$ give $y^0=P_C(0)=0$ and $z^1=0$.
The next iteration satisfies $z^2=y^1=z^1=0$, so the published stop returns
the exact solution after two outer iterations. All four projection calls
have feasible zero inputs and are handled by an exact feasibility shortcut.
The two-step weighted average is not zero; both it and the returned
auxiliary solution are recorded. This favourable structure is specific
to the example and is not a general speed claim.

FSPA-forward remains at distance approximately $0.229$ after 5000 steps.
The original point-to-set baseline used $0.02/(k+1)^{0.85}$, which does
satisfy the two series conditions. Its poor finite-budget result did not
establish asymptotic failure, but the routine was not the published
splitting method. The corrected comparison supports no claim of FSPA
superiority on this example.

\begin{figure}[tb]
\centering
\includegraphics[width=\textwidth]{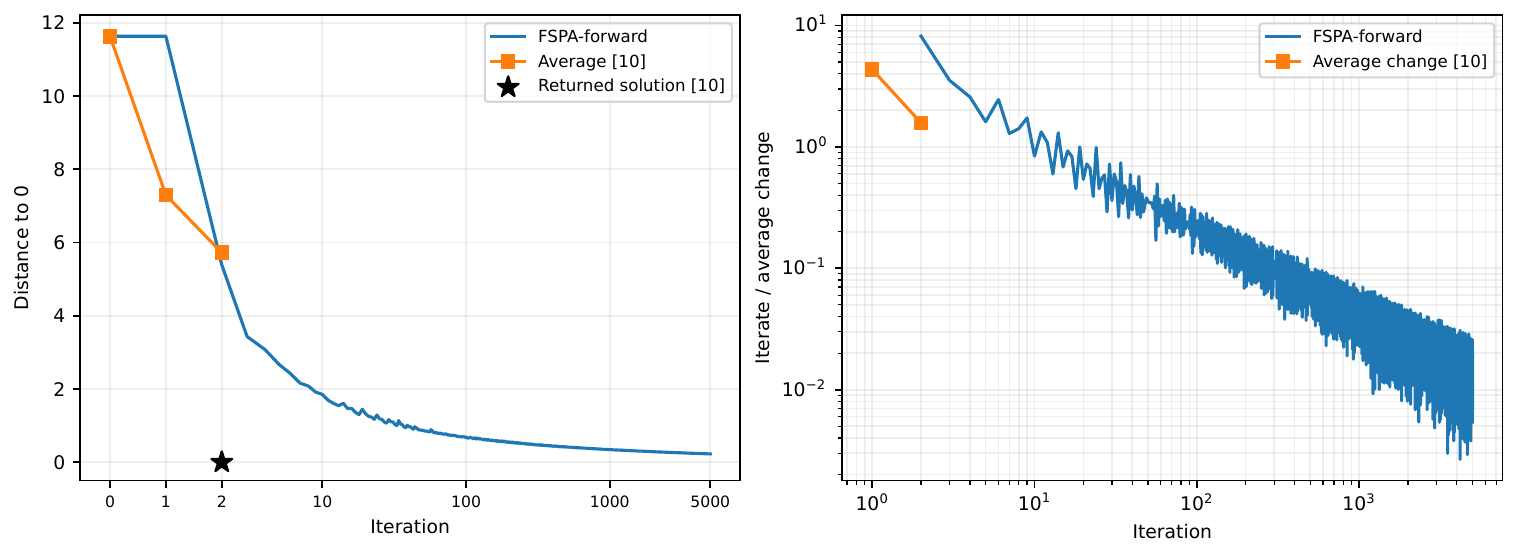}
\caption{Original point-to-set example. Left: distance of FSPA-forward
and the prescribed average in~\cite{BelloCruzDiazMillan2014} to $0$;
the separately marked point is the exact auxiliary solution returned at
iteration 2. The horizontal axis uses a linear scale through iteration 2
and a logarithmic scale thereafter. Right: FSPA iterate changes and changes
of the prescribed average until the fixed-point stop. No trajectory is
extended past termination. Zero changes are omitted on the logarithmic axis.}
\label{fig:pointset}
\end{figure}
\FloatBarrier

\section{Conclusion}\label{sec:conclusion}

FSPA integrates classical subgradient and Slater constructions with inherited
Haugazeau--Minty anchoring, an explicitly quantified feasible oracle, and a
point-to-set maximal-monotone analysis. Its outer update and feasibility
restoration require no $P_C$. Under the oracle assumption, the whole
sequence converges strongly to the nearest solution. The residual estimates
retain the averaged oracle error, and strong monotonicity yields a
best-iterate $O(N^{-1/4})$ distance bound when that average is
$O(N^{-1/2})$.

The oracle assumption is substantial. Certified constructions can involve
weighted projections or constrained global optimisation, and the concrete
constructions established here are finite-dimensional. No tested regime
combines certification with a demonstrated cost advantage over $P_C$-based
methods, and the practical forward-candidate variant remains heuristic.
The revised comparisons use published algorithms and their prescribed
output conventions. The disk experiment illustrates both the theoretical
envelope and the fixed-net error floor; the point-to-set example is solved
more effectively by the published splitting method.

Constructing application-specific Hilbert-space oracles and proving
convergence for inexpensive forward-candidate rules remain directions for
future work.

\section*{Acknowledgements}
The author thanks the referees for their constructive comments.

\end{document}